\documentclass[reqno, 11pt, a4paper]{amsart} %%%

\usepackage[T5,T1]{fontenc}
\usepackage{mlmodern}

\usepackage[utf8]{inputenc}
\allowdisplaybreaks
\usepackage{amsfonts}
\usepackage{amsmath}
\usepackage{amssymb}
\usepackage{amsthm}
\usepackage[text={33pc,605pt},centering, margin=1.25in]{geometry}     %%

\usepackage{mathrsfs} 
\usepackage[dvipsnames]{xcolor}

\usepackage{bbm}
\usepackage{mathtools}

\usepackage{dsfont}

\usepackage[pagebackref=true,colorlinks=true, linkcolor=Blue, citecolor=Blue, pdfencoding=auto, psdextra]{hyperref}
\renewcommand*\backref[1]{\ifx#1\relax \else (Cited on #1) \fi}

\usepackage{appendix}
\usepackage{enumitem}
\usepackage{float}

\usepackage{tikz}
\usepackage{tikz-cd} 
\usetikzlibrary{arrows, arrows.meta}

\usepackage[nameinlink]{cleveref} 
\usepackage{scalerel}[2016/12/29]

\theoremstyle{plain}
\newtheorem{definition}{Definition}
\newtheorem{proposition}[definition]{Proposition}
\newtheorem{lemma}[definition]{Lemma}

\newtheorem{theorem}[definition]{Theorem}
\newtheorem{remark}[definition]{Remark}

\theoremstyle{definition}

\numberwithin{definition}{section}
\numberwithin{equation}{section}

\newcommand*{\E}{\mathbb{E}}
\newcommand*{\R}{\mathbb{R}}

\newcommand*{\Z}{\mathbb{Z}}

\newcommand*{\N}{\mathbb{N}}

\newcommand{\abs}[1]{\left\lvert #1 \right\rvert}

\renewcommand*{\d}{\mathrm{d}}
\newcommand*{\e}{\mathrm{e}}

\makeatletter
\newcommand{\subalign}[1]{%
  \vcenter{%
    \Let@ \restore@math@cr \default@tag
    \baselineskip\fontdimen10 \scriptfont\tw@
    \advance\baselineskip\fontdimen12 \scriptfont\tw@
    \lineskip\thr@@\fontdimen8 \scriptfont\thr@@
    \lineskiplimit\lineskip
    \ialign{\hfil$\m@th\scriptstyle##$&$\m@th\scriptstyle{}##$\hfil\crcr
      #1\crcr
    }%
  }%
}
\makeatother

\crefname{equation}{}{}

\title[QFCLT for the RCM]{Quenched functional central limit theorem for the random conductance model under minimal moments}

\author[L.\ Kolesnikov]{Leonid Kolesnikov} %\orcidlink{0000-0002-3826-0479}}

\address[Leonid Kolesnikov]{TU Braunschweig, Institut für Mathematische Stochastik,
Germany.}
\email{leonid.kolesnikov@tu-braunschweig.de}

\author[Y. Steenbeck]{Yannic Steenbeck}

\address[Yannic Steenbeck]{TU Braunschweig, Institut für Mathematische Stochastik,
Germany.}
\email{yannic.steenbeck@tu-braunschweig.de}



\date{\today}

\usepackage{soul}

\begin{document}

\begin{abstract}
  We prove a quenched functional central limit theorem for the random conductance model with ergodic translation-invariant, strictly positive nearest-neighbor conductances, assuming only finite first moments of the conductances and their inverses. This settles an open problem by reaching the critical first-moment threshold, which is sharp for a certain class of integrability assumptions. A key ingredient of the proof is a new Sobolev inequality that seems to be missing from the literature.
\end{abstract}
\maketitle

\section{Quenched functional central limit theorem}

\subsection{The main result}

We consider the well-known \emph{random conductance model} that is given by the reversible \emph{variable-speed random walk} (VSRW) \(X = (X_t)_{t \geq 0}\) on \(\Z^d\), \(d \geq 1\), equipped with random strictly positive nearest-neighbor conductances \(\omega = (\omega(x, y))_{x, y \in \Z^d}\). We always assume that the conductances are symmetric in the sense that $\omega(x,y)=\omega(y,x)$ for all $x,y\in \Z^d$ and set $\omega(x,y)=0$ if $x$ is not a nearest-neighbor of $y$.

The conductances \(\omega\) are sampled from an ergodic translation-invariant probability distribution \(\mathbb{P}\) on \((\Omega, \mathcal{F})\) where \(\Omega = (0, \infty)^{E_d}\), for the set \(E_d = \{\{x, y\} \,\colon\, x,y \in \Z^d, \abs{x-y} = 1\}\) of unoriented nearest-neighbor edges of \(\Z^d\), and where \(\mathcal{F}\) is the Borel-$\sigma$-algebra on \(\Omega\). Here, translation-invariance of $\mathbb P$ means that $\mathbb{P}(\theta_z A) = \mathbb{P}(A)$ for every $A \in \mathcal{F}$ and every spatial shift \(\theta_z \colon \Omega \to \Omega\), \(z \in \mathbb{Z}^d\), defined by
\begin{align}
    (\theta_z \omega)(x, y)
    := \omega(x + z, y + z), \qquad {x,y} \in \Z^d.
\end{align} Ergodicity means then that every translation-invariant event has probability zero or one under \(\mathbb{P}\).

For $\omega\in \Omega$ fixed, the VSRW \(X\) is the continuous-time Markov chain with (formal) generator \(\mathcal{L}^\omega\), acting on functions \(f \colon \Z^d \to \R\) via 
\begin{align}
    (\mathcal{L}^\omega f)(x)
    = \sum_{y \sim x} \omega(x,y) [f(y) - f(x)], \quad x \in \Z^d,
\end{align} and (formal) Dirichlet form
\begin{align}
    \mathcal{E}^\omega(f, g)
    = \frac{1}{2} \sum_{x} \sum_{y \sim x} \omega(x,y) [f(y) - f(x)][g(y) - g(x)].
\end{align}

A rigorous construction of this continuous-time Markov process, as well as some further background on the random conductance model, can be found in the surveys of Biskup~\cite{Biskup2011Recent} and Andres~\cite{andres2025homogenization}.

For $x\in\mathbb Z^d$, we denote by $P_x^\omega$ the law of $X$ in the fixed environment $\omega$ (\emph{quenched law}), started at $x$, and by $E_x^\omega$ the corresponding expectation.

We say that a \emph{quenched functional central limit theorem (QFCLT)} holds for \(X\), if for \(\mathbb{P}\)-a.a.\ \(\omega\), under \(P_0^\omega\) the processes \((\frac{1}{n} X_{n^2 t})_{t \in [0, T]}\) converge as \(n \to \infty\) in distribution to \((\Sigma W_t)_{t \in [0, T]}\) for a standard Brownian motion \(W\) in \(\R^d\) and a deterministic covariance matrix \(\Sigma^2\), in the Skorokhod space \(D([0, T] \to \R^d)\), for every \(T > 0\). Here, $\Sigma$ denotes the unique symmetric positive semidefinite square root of the covariance matrix $\Sigma^2$.

We present the following main result.
\begin{theorem}[QFCLT]\label{theorem:quenched_functional_CLT}
    Let \(\mathbb{P}\) be an ergodic translation-invariant probability distribution on \((\Omega, \mathcal{F})\). Additionally, let \(\mu^\omega(x) := \sum_{y \sim x} \omega(x, y)\) and \(\nu^\omega(x) := \sum_{y \sim x} \omega(x, y)^{-1}\) for \(x \in \Z^d\), and suppose the moment condition
    \begin{align}
        \mathbb{E}[\mu^\omega(0)] < \infty, \quad \mathbb{E}[\nu^\omega(0)] < \infty,
    \end{align} holds.
    Then, a QFCLT holds for \(X\).
\end{theorem}
\begin{remark}[Covariance matrix]
   The inverse moment condition \(\mathbb{E}[\nu^\omega(0)] < \infty\) guarantees that \(\Sigma^2\) is non-degenerate. Moreover, if the conductance law \(\mathbb{P}\) is invariant under coordinate permutations and reflections, we have \(\Sigma^2 = \sigma^2 I_d\). 
   Semi-explicit formulas for \(\Sigma^2\) are known: a representation in terms of the so-called harmonic embedding is given in~\cite[Proposition 3.7]{andres2025homogenization}, and a corresponding variational formula can be derived from it, cf.~\cite{CaputoIoffe2003}.
\end{remark}

Naturally, numerous results in this direction have already been established in the literature. We give a very brief, non-exhaustive overview here:
\begin{enumerate}
\item For i.i.d.\ nearest-neighbor conductances in dimensions \(d\geq2\), the QFCLT is known~\cite{ABDH13} even when the conductances are allowed to vanish, provided that the probability of a positive conductance is supercritical for Bernoulli bond percolation, without requiring any finite-moment assumptions.

\item For ergodic conductances satisfying
\[
    \mathbb E[\mu^\omega(0)]<\infty
    \qquad\text{and}\qquad
    \mathbb E[\nu^\omega(0)]<\infty,
\]
the QFCLT is already known in dimensions \(d=1,2\); see~\cite[Section~4.4]{Biskup2011Recent}. Precursors to the one-dimensional result go back to Kawazu and Kesten~\cite{KawazuKesten1984}.

\item For ergodic conductances in dimensions \(d\geq3\), consider the \(M(p,q)\) moment conditions
\[
    \mathbb E[(\mu^\omega(0))^p]<\infty,
    \qquad
    \mathbb E[(\nu^\omega(0))^q]<\infty.
\]
Andres, Deuschel and Slowik~\cite{AndresDeuschelSlowik2015} showed that the QFCLT holds under
\[
    \frac{1}{p}+\frac{1}{q}<\frac{2}{d},
\]
which was subsequently improved by Bella and Schäffner~\cite{BellaSchaeffner2020} to
\[
    \frac{1}{p}+\frac{1}{q}<\frac{2}{d-1}.
\]
Our result answers affirmatively the open question of whether the QFCLT holds in all dimensions \(d\geq2\) under the moment condition $M(1,1)$. Moreover, the $M(1,1)$ condition is optimal in the following sense: for every $p<1$, there exist ergodic distributions satisfying $M(p,p)$ for which the QFCLT fails, see~\cite[Theorem 1.4]{BarlowBurdzyTimar2016}.

\item Beyond the setting considered here, QFCLTs have been established in several broader frameworks: for ergodic conductances that are allowed to vanish, QFCLTs are known under suitable moment and geometric regularity assumptions on the infinite cluster of positive-conductance edges; see, for instance,~\cite{DNS18} and, more recently,~\cite{ASS25}. In a different direction, quenched invariance principles have been obtained for ergodic random conductance models with long-range jumps under \(M(p,q)\)-type moment assumptions; see~\cite{BCKW21}. Finally, under additional integrability and regularity assumptions, in particular on the corrector, QFCLTs have also been established in the \(p\)-variation rough-path topology, allowing for degenerate conductances and long-range jumps; see~\cite{BBOS2026}.

\end{enumerate}

\subsection{Reduction and complete proof overview}

We marry, heavily supported by \texttt{AI}, the strategies of \cite{AndresDeuschelSlowik2015} and \cite{ba2015sobolev}, which in turn hinge on the so-called corrector method, which goes back (among others) to Kozlov~\cite{Kozlov1985}, and on Sobolev inequalities.
Notice that our proof is restricted to dimensions $d\geq 3$; the cases $d = 1, 2$ are already established in the literature and are not reproved here.

We assume from now on that the reader is familiar with the contents of \cite[Sections 3.1 and 3.2]{andres2025homogenization}. 
In particular, the Hilbert space $L^2_{\mathrm{cov}}$ consists of square-integrable scalar-valued covariant fields and admits the orthogonal decomposition $L^2_{\mathrm{cov}}=L^2_{\mathrm{pot}}\oplus L^2_{\mathrm{sol}}$. Here, $L^2_{\mathrm{pot}}$ is obtained as the closure of gradients of bounded local functions, while $L^2_{\mathrm{sol}}$ is its orthogonal complement. Applying the corresponding projection coordinatewise to the position field yields the $\mathbb R^d$-valued corrector $\chi=(\chi_1,\ldots,\chi_d)$. The corrected position $\Phi=(\Phi_1,\ldots,\Phi_d)$, given by $\Phi(\omega,x):=x-\chi(\omega,x)$, is the harmonic embedding; in particular, $x\mapsto\Phi(\omega,x)$ is $\mathcal{L}^\omega$-harmonic. 

Along the random walk $X$, this yields the decomposition $X_t=\Phi(\omega,X_t)+\chi(\omega,X_t)$, where $(\Phi(\omega,X_t))_{t\geq0}$ is a martingale and its diffusive scaling is handled by the standard martingale functional central limit theorem. The main difficulty is therefore to show that the corrector $\chi(\omega,X_t)$ is subdiffusive, which is the focus of our proof: We will recall the well-known, cf. \cite[(3.5)]{andres2025homogenization}, reduction of the validity of a QFCLT to the following 
\begin{proposition}[Quenched subdiffusivity of the corrector along paths in probability]\label{proposition:quenched_subdiffusivity_of_the_corrector_along_paths_in_probability}
    For \(\mathbb{P}\)-a.a.\ \(\omega \in \Omega\) and all \(T \geq 0\) it holds in \(P_0^\omega\)-probability that
    \begin{align}\label{equation:quenched_subdiffusivity_of_the_corrector_along_paths_in_probability}
        \frac{1}{n} \sup_{0 \leq t \leq n^2 T} \abs{\chi(\omega, X_t)} \xrightarrow[n \to \infty]{} 0.
    \end{align} 
\end{proposition}

To control the corrector, we follow the strategy of \cite{ba2015sobolev} and introduce a time-changed (sped-up) version \(Z\) of \(X\), which will have speed measure \(m^\omega\) given by a power of a Hardy-Littlewood maximal function of the inverse conductances \(\nu^\omega\). Notice that it indeed obviously suffices to show \eqref{equation:quenched_subdiffusivity_of_the_corrector_along_paths_in_probability} for any sped-up version \(Z\) of \(X\) to prove Proposition~\ref{proposition:quenched_subdiffusivity_of_the_corrector_along_paths_in_probability}.
Our \(Z\) will have its jump rates uniformly bounded from below by \(1\). Much better, a corresponding Sobolev inequality for the speed measure of \(Z\) will allow us to get a heat-kernel bound of the form
\begin{align}
    \frac{P_x^\omega(Z_t = y)}{m^\omega(y)} 
    \,\lesssim_d \, t^{-d/2},
\end{align} which heuristically tells us not only that \(Z\) jumps often enough, but that it also does not backtrack into the same, maybe atypical, regions too much. We will capitalize heavily on the corresponding rigorous statement.

We will now precisely construct the auxiliary random walk \(Z\) motivated by the Sobolev inequality we present in the following.
For a non-negative function \(h \colon \Z^d \to [0, \infty)\), consider the Hardy-Littlewood maximal function $\mathcal{M} h\colon \Z^d \to [0, \infty]$ given by
\begin{align}
    (\mathcal{M}h)(x)
    = \sup_{Q \ni x} \frac{1}{\abs{Q}} \sum_{y \in Q} h(y),
\end{align} where the supremum ranges over axis-parallel cubes. In general, \(Q\) will here always denote a cube of the form \(x + Q_R\), for \(x \in \Z^d\) and \(R \geq 0\), with \(Q_R := [-R, R]^{d} \cap \Z^d\) a lattice cube centered at the origin.
By the following standard weak-(1,1) lemma and our assumption \(\mathbb{E}[\nu^\omega(0)] < \infty\), we can make the choice \(h = 1 + \nu^\omega\).
\begin{lemma}\label{lemma:weak_1_1}
    Let \(h\) be a translation-invariant non-negative random field on \(\Z^d\) with \(\mathbb{E}[h(0)] < \infty\). Then, for all \(\lambda > 0\),
    \begin{align}
        \mathbb{P}(\mathcal{M}h(0) > \lambda)
        \lesssim_d \frac{\mathbb{E}[h(0)]}{\lambda}.
    \end{align}
    In particular, for \(\mathbb{P}\)-a.a.\ \(\omega\) we have \((\mathcal{M}h)(x) < \infty\) for all \(x \in \Z^d\).
\end{lemma}

We now present the following
\begin{lemma}[Crucial Sobolev inequality]\label{lemma:sobolev_inequality}
    Let \(p = \frac{2 d}{d-2}\) and \(\beta > \frac{p}{2}\).
    Suppose \(\omega \in \Omega\) is such that \(\mathcal{M}(\nu^\omega + 1) < \infty\) for all \(x \in \Z^d\). Denote \(m^\omega_\beta(x) := [\mathcal{M}(\nu^\omega + 1)(x)]^{-\beta}\).
    Then,
    \begin{align}
        \bigg(\sum_{x \in \Z^d} m^\omega_\beta(x) \abs{f(x)}^p \bigg)^{2/p}
        \lesssim_{d, \beta} \mathcal{E}^\omega(f, f)
    \end{align} for all finitely supported \(f \colon \Z^d \to \R\).
\end{lemma}
Its proof is given in Section~\ref{section:proofs_the_sobolev_inequality}.
We see that in particular \(\beta = 4\) works for all \(d \geq 3\) in the above lemma. This motivates setting
\begin{align}
    m^\omega := m^\omega_4 = [\mathcal{M}(\nu^\omega + 1)]^{-4}
\end{align} and considering the auxiliary random walk \(Z\) with rates
\begin{align}
    \lambda^\omega(x, y) := \frac{\omega(x, y)}{m^\omega(x)}.
\end{align}
To be technically precise, we first restrict to the (translation-invariant, probability \(1\)) set of \(\omega \in \Omega\) such that \(\mathcal{M}(\nu^\omega + 1) < \infty\) for all \(x \in \Z^d\), and just set \(m^\omega = 1\) on its (probability \(0\)) complement.
We denote the corresponding quenched laws and expectations by \(P_x^\omega\) and \(E_x^\omega\), \(x \in \Z^d\), just as for the physical random walk \(X\). We denote
\begin{align}
    p_t^\omega(x, y)
    = P_x^\omega(Z_t = y), 
    \quad q_t^\omega(x, y) 
    = \frac{p_t^\omega(x, y)}{m^\omega(y)}, \quad x,y \in \Z^d,
\end{align} for the auxiliary walk \(Z\).
Notice that the sped-up walk \(Z\) remains non-explosive, since it is obtained from the non-explosive walk \(X\) by the standard time-change with speed measure \(m^\omega\), where \(0 < \mathbb{E}[m^\omega(0)] < \infty\), and thus \(\int_0^t m^\omega(X_s) \d s \xrightarrow[t \to \infty]{} \infty\); cf.~\cite[Remark~3.4(ii)]{andres2025homogenization}.

We will later, in Section~\ref{section:proofs_uniform_smoothing}, see by known arguments that the Sobolev inequality indeed implies heat-kernel control as in the following lemma. For a finite set \(D \Subset \Z^d\) we will always write \(\tau_D\) for the first exit time from \(D\) of the random walk under consideration. The heat kernel of the auxiliary walk \(Z\) killed on exiting \(D\) is denoted by
\begin{align}
    q_t^{\omega, D}(x, y)
    = \frac{P_x^\omega(Z_t = y, \tau_D > t)}{m^\omega(y)}, 
    \quad x, y \in D.
\end{align}
\begin{lemma}[Uniform smoothing]\label{lemma:uniform_smoothing_heat_kernel_density_bound}
    For every finite set \(D \Subset \Z^d\), the auxiliary walk \(Z\) killed on exiting \(D\) obeys 
    \begin{align}\label{equation:uniform_smoothing_heat_kernel_density_bound}
        q_t^{\omega, D}(x, y)
        \lesssim_d t^{-d/2}, \quad t  > 0.
    \end{align} 
    In particular, the same bound holds for \(q_t^\omega\) in place of \(q_t^{\omega, D}\).
\end{lemma}

Turning back to a proof of Proposition~\ref{proposition:quenched_subdiffusivity_of_the_corrector_along_paths_in_probability}, we recall first that it suffices to show the same convergence with \(Z\) in place of \(X\), because \(Z\) is just a sped-up version of \(X\). 
It will come in handy for the next steps to introduce some further notation. For a field \(r \colon \Omega \times \Z^d \times \Z^d \to \R\), denote by
\begin{align}
    e_r^\omega(x)
    = \sum_{y \sim x} \omega(x, y) \abs{r(\omega, x, y)}^2, \quad x \in \Z^d,
\end{align} its \emph{physical conductance energy}, by 
\begin{align}
    \mathcal{M}_0 e^\omega_r
    = (\mathcal{M} e^\omega_r)(0)
    %\sup_{R \geq 0} \frac{1}{\abs{Q_R}} \sum_{x \in Q_R} e_r^\omega(x) 
\end{align} its maximal average at the origin, and its \emph{path current} by
\begin{align}
    J^\omega_t(r)
    = \sum_{0 < s \leq t \colon Z_s \neq Z_{s-}} r(\omega, Z_{s-}, Z_s).
\end{align}
Now we recall that by definition of the corrector, we can approximate \(\chi\) in \(L^2_{\mathrm{cov}}\) by gradients of bounded functions \(\phi^{(k)}\). This means that there are bounded measurable functions \(\phi^{(k)} \colon \Omega \to \R^d\) such that for the approximation error
\begin{align}\label{equation:approximation_error_r_k}
    r^{(k)}(\omega, x, y)
    := [\chi(\omega, y) - \chi(\omega, x)] - [\phi^{(k)}(\theta_y \omega) - \phi^{(k)}(\theta_x \omega)]
\end{align} it holds
\begin{align}\label{equation:energy_goes_to_zero}
    \mathbb{E}[e^\omega_{r^{(k)}_j}(0)] = \Vert r^{(k)}_j(\cdot, 0, \cdot) \Vert_{L^2_{\mathrm{cov}}}^2
    \xrightarrow[k \to \infty]{} 0, \quad 1 \leq j \leq d,
\end{align} coordinatewise. One can perhaps already sense now, thinking of \eqref{equation:energy_goes_to_zero} and the weak-(1,1) bound Lemma~\ref{lemma:weak_1_1}, that the following proposition closes the proof of Proposition~\ref{proposition:quenched_subdiffusivity_of_the_corrector_along_paths_in_probability}. This Proposition~\ref{proposition:subdiffusive_control_by_maximal_function_of_energy} vaguely embodies the idea that \emph{fluctuations cost energy} and reminds us in spirit of \cite{AnconaLyonsPeres1999}.
A proof of the reduction of Proposition~\ref{proposition:quenched_subdiffusivity_of_the_corrector_along_paths_in_probability} to it can be found in Section~\ref{section:proofs_reduction}.
\begin{proposition}\label{proposition:subdiffusive_control_by_maximal_function_of_energy}
    There exists a finite random variable \(A^\omega\) and a finite constant \(B_d\), such that for every measurable, covariant and antisymmetric field \(r \colon \Omega \times \Z^d \times \Z^d \to \R\) and \(\epsilon > 0\), \(K \geq 1\), it holds that
    \begin{align}
        \limsup_{t \to \infty} P_0^\omega \big( \sup_{s \leq t} \abs{J_s^\omega(r)} > \epsilon \sqrt{t} \big)
        \leq \frac{A^\omega}{K} + \frac{B_d K^{d/2}}{\epsilon} \sqrt{\mathcal{M}_0(e_r^\omega)}.
    \end{align}
\end{proposition}
In turn, the proof of Proposition~\ref{proposition:subdiffusive_control_by_maximal_function_of_energy} hinges on the following two lemmas, which will be proven in Section~\ref{section:diffusive_range} resp.\ Section~\ref{section:confined_path_currents}.
\begin{lemma}[Diffusive range]\label{lemma:auxiliary_RW_diffusive_range}
    There is a finite random variable \(C^\omega\) such that for all \(t \geq 1\) it holds
    \begin{align}
        E_0^\omega \Big[ \sup_{s \leq t} \abs{Z_s} \Big]
        \leq C^\omega \sqrt{t}.
    \end{align}
\end{lemma}
\begin{lemma}[Confined path currents]\label{lemma:auxiliary_RW_confined_currents}
        In the setting of Proposition~\ref{proposition:subdiffusive_control_by_maximal_function_of_energy}, it holds for all \(s \geq 2\) and \(L \geq 1\) that
        \begin{align}
            E_0^\omega \Big[\mathbf{1}_{\tau_{Q_{\lceil L \sqrt{s} \rceil}} > s}  \sup_{s/2 \leq v \leq s} \abs{J^\omega_{v}(r) - J^\omega_{s/2}(r)} \Big]
            \leq C_d \sqrt{s L^d} \sqrt{\mathcal{M}_0(e_r^\omega)}.
        \end{align}
    \end{lemma}
Conditional on these two lemmas, the elementary idea is to control the path current \(J(r)\) by combining 
\begin{enumerate}
    \item control of path current increments \(\sup_{s/2 \leq v \leq s} \abs{J^\omega_{v}(r) - J^\omega_{s/2}(r)}\) for \(Z\) confined inside boxes \(Q_{\lceil L \sqrt{s} \rceil}\) of diffusive scaling given by Lemma~\ref{lemma:auxiliary_RW_confined_currents},
    
    \item and diffusive range of the auxiliary random walk \(Z\) given by Lemma~\ref{lemma:auxiliary_RW_diffusive_range},
\end{enumerate} by dyadic decomposition of time intervals. The full proof of Proposition~\ref{proposition:subdiffusive_control_by_maximal_function_of_energy} is given in Section~\ref{section:proofs_reduction}.

Let us finally showcase the ideas and techniques behind Lemma~\ref{lemma:auxiliary_RW_diffusive_range} and Lemma~\ref{lemma:auxiliary_RW_confined_currents} in the following two sections.

\subsubsection{Diffusive range}

The basic mantra that we will exploit to control the range of our auxiliary random walk \(Z\) as in Lemma~\ref{lemma:auxiliary_RW_diffusive_range} is
\begin{align*}
    \text{motion} \lesssim \text{random fluctuations} + \text{directed transport}.
\end{align*}
Let us first discuss the quantities that will appear on the right-hand side of such an equation in our context, namely activity and entropy.

Denote by \(A^\omega(t)\) the expected \emph{activity}
\begin{align}
    A^\omega(t) := E_0^\omega[N_t],
\end{align} where \(N_t\) is the number of jumps of \(Z\) in the time interval \([0, t]\).
We have the following almost-sure bound on the activity.
\begin{lemma}\label{lemma:almost_sure_activity_bound}
    There is a \(\mathbb{P}\)-a.s.\ finite random variable \(K^\omega\) such that
    \begin{align}\label{equation:activity_bound}
        A^\omega(t) \leq K^\omega t, \quad t \geq 1.
    \end{align}
\end{lemma}

Moving on, we consider \(H^\omega(t)\), the negative of the relative entropy  of \(p^\omega_t\) w.r.t. the speed measure \(m^\omega\), in the following just \emph{entropy}, given by
\begin{align}
    H^\omega(t)
    := - \sum_{x \in \Z^d} p^\omega_t(x) \log q^\omega_t(x)
\end{align} where we denote here \(p^\omega_t(x) := p^\omega_t(0, x)\) and \(q^\omega_t(x) := q^\omega_t(0, x) = \frac{p^\omega_t(x)}{m^\omega(x)}\).
This quantity is well-defined in our setting, and interpreted as weighted spread of the distribution \(p^\omega_t\).
\begin{lemma}[Entropy and de Bruijn-type identity]
    The entropy \(H^\omega(t)\) is a finite real number for every \(t \geq 0\) and fulfills \(H^\omega(0) = \log m^\omega(0)\) as well as a \emph{de Bruijn-type identity}
    \begin{align}
        H^\omega(t) - H^\omega(a)
        = \int_a^t \, I^\omega(s) \, \d s, \quad 0 < a < t,
    \end{align} with the non-negative \emph{entropy production}
    \begin{align}
        I^\omega(s)
        = \mathcal{E}^\omega(q^\omega_s, \log q^\omega_s).
    \end{align}
\end{lemma}

Together, activity and entropy control the movement of \(Z\) as follows.
\begin{lemma}[Range-entropy-activity inequality]\label{lemma:range_entropy_activity_inequality}
    For all \(0 < a < t\) it holds that
    \begin{align}\label{equation:range_entropy_activity_inequality}
        E_0^\omega[\sup_{a \leq s \leq t} \abs{Z_s - Z_a}]
        \leq \bigg(3 + \sqrt{[H^\omega(t) - H^\omega(a)]/2} \bigg) \sqrt{A^\omega(t) - A^\omega(a)}.
    \end{align}
\end{lemma}
A proof of Lemma~\ref{lemma:range_entropy_activity_inequality} will appear in a separate note.

Now, we already saw in Lemma~\ref{lemma:almost_sure_activity_bound} that up to random constants, the activity behaves like
\begin{align*}
    A^\omega(t) \lesssim t,
\end{align*} so that it remains to control the entropy to get from Lemma~\ref{lemma:range_entropy_activity_inequality} to the bound in Lemma~\ref{lemma:auxiliary_RW_diffusive_range}.
We will show that up to random constants, the entropy behaves like \(\frac{d}{2} \log(t)\). This is to be expected heuristically, as \(p^\omega_t(x) / m^\omega(x)\) should be roughly uniform on a volume of \((\sqrt{t}^d) = t^{d/2}\) because the random walk spreads over a distance \(\sqrt{t}\), so that
\begin{align*}
    -\sum_x p^\omega_t(x) \log \frac{p^\omega_t(x)}{m^\omega(x)}
    \approx \frac{d}{2} \sum_x  p^\omega_t(x) \log(t)
    = \frac{d}{2} \log(t).
\end{align*} A corresponding rigorous lemma can be found as Lemma~\ref{lemma:entropy_comparison}.
Hence, there would be a disruptive \(\sqrt{\log(t)}\)-factor if we naively applied Lemma~\ref{lemma:range_entropy_activity_inequality} with \(a = 0\). Luckily, \(\log(t) - \log(t/2) = \log(2)\), which will allow us in Section~\ref{section:diffusive_range} to bootstrap \eqref{equation:range_entropy_activity_inequality} to a proof of Lemma~\ref{lemma:auxiliary_RW_diffusive_range}.

\subsubsection{Confined path currents are energy-controlled}

The proof of Lemma~\ref{lemma:auxiliary_RW_confined_currents} rests on the following argument.
First, we bound fluctuations along random walk paths not for the walk which starts at the origin \(0\) and is killed upon leaving the cube \(Q\), but the correspondingly reflected random walk, i.e. suppressing jumps to the outside of \(Q\), started from its reversible probability measure \(\pi^\omega_{Q}(x) = \frac{m^\omega(x)}{m^\omega(Q)}\). 
\begin{lemma}[Fluctuations-energy bound for the reflected walk in equilibrium]\label{lemma:fluctuations_energy_for_reflected_walk_in_equilibrium}
    For all antisymmetric measurable \(r \colon \Omega \times \Z^d \times \Z^d \to \R\) and all \(h \geq 0\), we have
    \begin{align}\label{equation:fluctuations_energy_for_reflected_walk_in_equilibrium}
        E^{\mathrm{reflected}}_{\pi_Q^\omega}\big[ \sup_{s \leq h} \abs{J^\omega_s(r)}^2 \big]
        \lesssim \frac{h}{m^\omega(Q)} \sum_{x, y \in Q} \omega(x, y) \abs{r(\omega, x, y)}^2.
    \end{align}
\end{lemma}
The inequality \eqref{equation:fluctuations_energy_for_reflected_walk_in_equilibrium} can be proved via a forward-backward martingale decomposition of \(J^\omega(r)\) under the reversible process described above, where drifts cancel because \(r\) is antisymmetric, and Doob's martingale inequality.

Let us explain shortly how this result transfers to the energy-control over fluctuations along paths of the killed random walk started at the origin \(0\), from time \(s/2\) to time \(s\), given in Lemma~\ref{lemma:auxiliary_RW_confined_currents}. This is possible via the Markov property at time \(s/2\) and the diffusivity \(P_0^\omega(Z_{s/2} = x, \tau_Q > s/2) \lesssim_d s^{-d/2} m^\omega(x)\) over the invariant measure \(\pi^\omega_{Q}(x) = \frac{m^\omega(x)}{m^\omega(Q)}\) from Lemma~\ref{lemma:uniform_smoothing_heat_kernel_density_bound} which could already develop until the time \(s/2\).
A rigorous implementation of this idea is recorded in Section~\ref{section:confined_path_currents}.

\section{Proofs}

\subsection{The Sobolev inequality}\label{section:proofs_the_sobolev_inequality}

The Sobolev inequality Lemma~\ref{lemma:sobolev_inequality} is the appropriate analogue of \cite[Theorem 2]{ba2015sobolev} in our discrete and infinite-volume setting. We split the proof into the following small lemmas.

First, we get the following
\begin{lemma}[Pointwise control]
    In the setting of Lemma~\ref{lemma:sobolev_inequality},
    \begin{align}\label{equation:pointwise_control_sobolev}
        \abs{f(x)}
        \lesssim_d [\mathcal{M}(1+\nu^\omega)(x)]^{1/2} \mathcal{E}^\omega(f, f)^{1/d} [\mathcal{M}e_f^\omega(x)]^{1/p}.
    \end{align}
\end{lemma}
\begin{proof}
    For the sake of this proof, denote
    \begin{align*}
        a_f(z) = \sum_{y \sim z} \abs{f(y) - f(z)}, 
        \quad e^\omega_f(z) = \sum_{y \sim z} \omega(z, y) \abs{f(z)-f(y)}^2,
        \quad \mathrm{avg}_Q g = \frac{1}{\abs{Q}} \sum_{z \in Q} g(z).
    \end{align*}
    First notice that for the dyadic cubes \(\mathfrak{Q}_j = x + Q_{2^j}\), we have by telescoping, \(\mathrm{avg}_{\mathfrak{Q}_j} f \xrightarrow[j \to \infty]{} 0\) for finitely supported \(f\), and the discrete \(L^1\)-Poincaré inequality \cite[Proposition 12.3]{HajlaszKoskela2000}, the bound
    \begin{align}
        \abs{f(x)}
        \lesssim_d \sum_{j \geq 0} 2^j \mathrm{avg}_{\mathfrak{Q}_j} a_f.
    \end{align}
    Applying Cauchy-Schwarz first at a vertex and then in the cube, we get
    \begin{align}
        \mathrm{avg}_{Q} a_f
        \leq \sqrt{\mathrm{avg}_{Q} \nu^\omega} \sqrt{\mathrm{avg}_{Q} e_f^\omega}.
    \end{align}
    It follows that
    \begin{align}
        \abs{f(x)}
        &\lesssim_d \sum_{j \geq 0} 2^j \sqrt{\mathrm{avg}_{\mathfrak{Q}_j} \nu^\omega} \sqrt{\mathrm{avg}_{\mathfrak{Q}_j} e_f^\omega} \\
        &\leq \sqrt{\mathcal{M}(\nu^\omega + 1)(x)} \sum_{j \geq 0} 2^j \sqrt{\min\{\mathcal{M}e_f^\omega(x), 2^{-jd +1}\mathcal{E}^\omega(f, f) \}} \nonumber\\
        &\lesssim_d  \sqrt{\mathcal{M}(\nu^\omega + 1)(x)} \mathcal{E}^\omega(f, f)^{1/d} [\mathcal{M}e_f^\omega(x)]^{1/p}. \nonumber
    \end{align}
\end{proof}

From this, it follows a
\begin{lemma}[Weak Sobolev inequality]\label{lemma:weak_sobolev_inequality}
    In the setting of Lemma~\ref{lemma:sobolev_inequality},
    \begin{align}\label{equation:weak_sobolev_inequality}
         m^\omega_\beta(\{\abs{f} > \lambda \})
        \lesssim_{\beta, d} \lambda^{-p} \, \mathcal{E}^\omega(f, f)^{p/2}.
    \end{align}
\end{lemma}
\begin{proof}
    Indeed, by \eqref{equation:pointwise_control_sobolev}, we have, using the weak-(1,1) maximal inequality for counting measure,
    \begin{align}
        &m_\beta^\omega(\{\abs{f} > \lambda \})
        \leq m_\beta^\omega(\{\mathcal{M}e_f^\omega > C_d \lambda^p \mathcal{E}^\omega(f,f)^{-p/d} (m_1^\omega)^{p/2} \}) \\
        &\leq \sum_{j \geq 0} \sum_{x \colon m^\omega_1(x) \in (2^{-j-1}, 2^{-j}]} m^\omega_\beta(x) \mathbf{1}_{\mathcal{M}e^\omega_f(x) > C_d \lambda^p \mathcal{E}^\omega(f,f)^{-p/d} 2^{-(j+1)p/2}} \nonumber\\
        &\lesssim_d \sum_{j \geq 0} 2^{-j(\beta - \frac{p}{2})} \lambda^{-p} \mathcal{E}^\omega(f,f)^{p/d} \sum_{x \in \Z^d}  e_f^\omega(x) \nonumber\\
        &\lesssim_{\beta} \lambda^{-p} \mathcal{E}^\omega(f,f)^{p/d + 1}
        = \lambda^{-p} \mathcal{E}^\omega(f,f)^{p/2}. \nonumber
    \end{align}
\end{proof}

We conclude with a
\begin{proof}[Proof of Lemma~\ref{lemma:sobolev_inequality}]
    As presented in \cite[Theorem 1, Theorem 2]{Hajlasz2001Sobolev}, the classical Maz'ya truncation argument allows us to upgrade the weak Sobolev inequality from Lemma~\ref{lemma:weak_sobolev_inequality} to a strong one.
\end{proof}
\begin{remark}
    The strong Sobolev inequality is actually not necessary for our application. Indeed, the Nash inequality \eqref{equation:nash_inequality} already follows from the weak Sobolev inequality \eqref{equation:weak_sobolev_inequality} via the layer-cake representation of \(\Vert f \Vert_{\ell^2(m^\omega)}^2\) as
    \begin{align}
        &\Vert f \Vert_{\ell^2(m^\omega)}^2
        = 2 \int_0^\infty  \lambda \, m^\omega(\{\abs{f} > \lambda\}) \,  \d\lambda 
        \lesssim_{d} \int_0^\infty  \lambda \, \min\!\big\{\lambda^{-1} \Vert f\Vert_{\ell^1(m^\omega)},  \lambda^{-p} [\mathcal{E}^\omega(f, f)]^{p/2} \big\} \, \d\lambda \\
        &=  \Vert f\Vert_{\ell^1(m^\omega)} \int_0^{\lambda^*} \d\lambda + [\mathcal{E}^\omega(f, f)]^{p/2}\int_{\lambda^*}^{\infty} \lambda^{1-p} \,\d\lambda
        \lesssim  \Vert f \Vert_{\ell^1(m^\omega)}^{4/(d+2)} [\mathcal{E}^\omega(f, f)]^{d/(d+2)}  \nonumber
    \end{align} where \(\lambda^* = \big[\mathcal{E}^\omega(f,f)^{p/2} / \Vert f\Vert_{\ell^1(m^\omega)} \big]^{1/(p-1)}\) is the intersection point of the two bounds for the integrand.
\end{remark}

\subsection{Uniform smoothing}\label{section:proofs_uniform_smoothing}

We recall an essentially classical functional inequalities argument on how to turn a Sobolev inequality such as Lemma~\ref{lemma:sobolev_inequality} into a heat kernel density bound such as Lemma~\ref{lemma:uniform_smoothing_heat_kernel_density_bound}.
\begin{proof}[Proof of Lemma~\ref{lemma:uniform_smoothing_heat_kernel_density_bound}]
    The Hölder interpolation 
    \begin{align*}
        \Vert f \Vert_{\ell^2(m^\omega)}^2
        \leq \Vert f \Vert_{\ell^p(m^\omega)}^{2d/(d+2)} \Vert f \Vert_{\ell^1(m^\omega)}^{4/(d+2)}
    \end{align*}
    turns the Sobolev inequality \(\Vert f\Vert_{\ell^p(m^\omega)}^2 \lesssim_d \mathcal{E}^\omega(f, f)\) from Lemma~\ref{lemma:sobolev_inequality} into the Nash inequality
    \begin{align}\label{equation:nash_inequality}
        \Vert f\Vert^{2 + 4/d}_{\ell^2(m^\omega)} \Vert f\Vert^{-4/d}_{\ell^1(m^\omega)}
        \lesssim_d \mathcal{E}^\omega(f, f).
    \end{align}
    The uniform heat-kernel density bound \eqref{equation:uniform_smoothing_heat_kernel_density_bound} follows from that as in~\cite[Lemma 2.8]{BarlowDeuschel2010}.
\end{proof}

\subsection{Reduction to diffusive range and confined path current inequality}\label{section:proofs_reduction}

We already argued that Proposition~\ref{proposition:quenched_subdiffusivity_of_the_corrector_along_paths_in_probability} (and hence Theorem~\ref{theorem:quenched_functional_CLT}) follows from the corresponding statement for the auxiliary walk \(Z\). Now we supply the already indicated proof conditional on Proposition~\ref{proposition:subdiffusive_control_by_maximal_function_of_energy}. A proof of the latter is provided directly after.
\begin{proposition}\label{proposition:quenched_sudiffusivity_of_the_corrector_in_probability_along_auxiliary_walk}
    For \(\mathbb{P}\)-a.a.\ \(\omega \in \Omega\) and all \(T \geq 0\) it holds in \(P_0^\omega\)-probability that
\begin{align}\label{equation:quenched_sudiffusivity_of_the_corrector_in_probability_along_auxiliary_walk}
        \frac{1}{n} \sup_{0 \leq t \leq n^2 T} \abs{\chi(\omega, Z_t)} \xrightarrow[n \to \infty]{} 0.
    \end{align} 
\end{proposition}
\begin{proof}
    Recall the approximation of \(\chi\) by the gradients of bounded measurable \((\phi^{(k)})_{k \in \N}\) and the corresponding approximation errors \((r^{(k)})_{k \in \N}\) from \eqref{equation:approximation_error_r_k}.
    We get from \(\chi(\omega, 0) = 0\) and telescoping
    \begin{align}
        \chi(\omega, Z_t)
        = J_t^{\omega}(r^{(k)}) + [\phi^{(k)}(\theta_{Z_t} \omega) - \phi^{(k)}(\omega)].
    \end{align} The second term on the r.h.s. is bounded by \(2 \Vert \phi^{(k)} \Vert_\infty\), so that we only have to control the first one if we can send \(t \to \infty\) first and \(k \to \infty\) later.
    Indeed, we have for all \(\epsilon > 0\) and \(K \geq 1\),
    \begin{align}
        \limsup_{t \to \infty} P_0^\omega\big(\sup_{s \leq t} \abs{\chi(\omega, Z_s)} > \epsilon \sqrt{t}\big)
        &\leq 
        \sum_{j = 1}^{d} \limsup_{t \to \infty} P_0^\omega\big(\sup_{s \leq t} \abs{J^\omega_s(r^{(k)}_j)} >  \frac{\epsilon}{c_d} \sqrt{t}\big) \\
        &\leq \sum_{j = 1}^{d} \bigg[ \frac{A^\omega}{K} + \frac{c_d B_d K^{d/2}}{\epsilon}  \sqrt{\mathcal{M}_0(e^\omega_{r^{(k)}_j})} \bigg] \nonumber
    \end{align} by Proposition~\ref{proposition:subdiffusive_control_by_maximal_function_of_energy}. Now send first \(k \to \infty\) along a suitable subsequence, using the independence of the left-hand side of \(k\) and \eqref{equation:energy_goes_to_zero}, and then \(K \to \infty\) to obtain that it \(\mathbb{P}\)-a.s.\ holds that
    \begin{align}
        \limsup_{t \to \infty} P_0^\omega\big(\sup_{s \leq t} \abs{\chi(\omega, Z_s)} > \epsilon \sqrt{t}\big) = 0
    \end{align} for all \(\epsilon > 0\). Measurability and the right order of quantifiers can be ensured by the standard arguments of choosing rational \(\epsilon\) and so on.
\end{proof}

As promised, we will now demonstrate the elementary reduction of Proposition~\ref{proposition:subdiffusive_control_by_maximal_function_of_energy} to Lemma~\ref{lemma:auxiliary_RW_diffusive_range} and Lemma~\ref{lemma:auxiliary_RW_confined_currents}.
\begin{proof}[Proof of Proposition~\ref{proposition:subdiffusive_control_by_maximal_function_of_energy}]
    Put \(s_j = t 2^{-j}\) and \(L_j = K 2^{j/(2d)}\) for \(t \geq 2\). Define a good event 
    \begin{align}
        G_t
        := \{\forall \, 0 \leq j < \lfloor \log_2 t \rfloor \,\colon\, \tau_{Q_{\lceil L_j \sqrt{s_j} \rceil}} > s_j \big\}.
    \end{align}
    The probability of \(G_t\) not happening is controlled by Lemma~\ref{lemma:auxiliary_RW_diffusive_range} and a union bound as
    \begin{align}
        P^\omega_0(G_t^c) 
        \leq \frac{C^\omega}{K} \sum_{j \geq 0} 2^{-j/(2d)}.
    \end{align}
    On \(G_t\), we will use the brutal deterministic bound
    \begin{align}
        \sup_{s \leq t} \abs{J_s^\omega(r)}
        &\leq \sup_{s \leq 2} \abs{J_s^\omega(r)} + \sum_{0 \leq j < \lfloor \log_2 t \rfloor} \sup_{s_j/2 \leq s \leq s_j} \abs{J^\omega_s(r) - J^\omega_{s_j / 2}(r)}.
    \end{align} The first summand on the r.h.s.\ is finite almost-surely and independent of \(t\), so that it vanishes under the \(\frac{1}{\sqrt{t}}\)-scaling as \(t \to \infty\).
    For the second summand we apply Lemma~\ref{lemma:auxiliary_RW_confined_currents} after Markov's inequality to get
    \begin{align}
        &P_0^\omega\Big( \Big\{\sum_{0 \leq j < \lfloor \log_2 t \rfloor} \sup_{s_j/2 \leq s \leq s_j} \abs{J^\omega_s(r) - J^\omega_{s_j / 2}(r)} > \epsilon \sqrt{t}\Big\} \cap G_t \Big) \\
        &\lesssim_d \sum_{j \geq 0} \frac{1}{\epsilon \sqrt{t}} \sqrt{s_j L_j^d} \sqrt{\mathcal{M}_0(e^\omega_r)}
        \leq \frac{K^{d/2}}{\epsilon} \sqrt{\mathcal{M}_0(e^\omega_r)} \sum_{j \geq 0} 2^{-j/4}. \nonumber
    \end{align}
\end{proof}

\subsection{Diffusive range}\label{section:diffusive_range}

Here, we present a
\begin{proof}[Proof of Lemma~\ref{lemma:auxiliary_RW_diffusive_range}]
    We bootstrap the inequality \eqref{equation:range_entropy_activity_inequality}. 
    Set \(U(t) := \frac{1}{\sqrt{t}} E_0^\omega[\sup_{s \leq t} \abs{Z_s}]\). The entropy comparison \eqref{equation:entropy_difference_half_the_time} gives
    \begin{align}
        H^\omega(t) - H^\omega(t/2) 
        \leq C_d'' + d \log(1 + U(t)).
    \end{align} 
    The pathwise inequality
    \begin{align}
        \sup_{s \leq t} \abs{Z_s}
        \leq \sup_{s \leq t/2} \abs{Z_s} + \sup_{t/2 \leq s \leq t} \abs{Z_s - Z_{t/2}}
    \end{align} together with the range-entropy-activity inequality \eqref{equation:range_entropy_activity_inequality} and the activity bound \eqref{equation:activity_bound} yields
    \begin{align}
        U(t)
        \leq \frac{1}{\sqrt{2}} U(t/2) + \widetilde{K}^\omega \sqrt{1 + \log(1 + U(t))}.
    \end{align}
    Now, taking the supremum, we get
    \begin{align}\label{equation:bootstrap}
        \sup_{1 \leq t \leq T} U(t)
        \leq \max\Big\{ A^\omega(2), \frac{1}{\sqrt{2}} \sup_{1 \leq t \leq T} U(t) +  \widetilde{K}^\omega \sqrt{1 + \log[1 + \sup_{1 \leq t \leq T} U(t)]} \Big\}
    \end{align} because nearest-neighbor jumps immediately give \(U(t) \leq \frac{1}{\sqrt{t}} A^\omega(t)\).
    
    But if \(\sup_{1 \leq t \leq T} U(t) > \max\{A^\omega(2), 1\}\), then \eqref{equation:bootstrap} and \(\log(1+z) \leq z\) already imply
    \begin{align}
        \sup_{1 \leq t \leq T} U(t)
        \leq \frac{\sqrt{2}}{1 - \frac{1}{\sqrt{2}}} \widetilde{K}^\omega \sqrt{\sup_{1 \leq t \leq T} U(t)}
    \end{align} and consequently
    \begin{align}
        \sup_{1 \leq t \leq T} U(t)
        \leq \max\{A^\omega(2), 1,  \frac{2}{(1 - \frac{1}{\sqrt{2}})^2} (\widetilde{K}^\omega)^2\}
    \end{align} uniformly in \(T \geq 2\).
\end{proof}

\begin{lemma}[\(H^\omega(t) \approx \frac{d}{2} \log(t)\) up to a posteriori bounded terms]\label{lemma:entropy_comparison}
    We have, for \(t \geq 1\),
    \begin{align}\label{equation:entropy_comparison}
        \frac{d}{2} \log(t) - C_d
        \leq H^\omega(t)
        \leq \frac{d}{2} \log(t) + C_d' + d \log(1 + t^{-1/2} E_0^\omega[\abs{Z_t}]).
    \end{align}
    In particular, for \(t \geq 2\),
    \begin{align}\label{equation:entropy_difference_half_the_time}
        H^\omega(t) - H^\omega(t/2) 
        \leq C_d'' + d \log(1 + t^{-1/2} E_0^\omega[\abs{Z_t}]).
    \end{align}
\end{lemma}
\begin{proof}
    The lower bound on \(H^\omega(t)\) follows immediately from the heat-kernel density bound Lemma~\ref{lemma:uniform_smoothing_heat_kernel_density_bound}.
    For the upper bound, we compare with the Gibbsian probability density \(Z_\alpha^{-1} \e^{-\alpha \abs{x}}\).
    Indeed, since \(m^\omega \leq 1\),
    \begin{align}
        H^\omega(t)
        \leq \sum_{x \in \Z^d} p^\omega_t(x) \log\Big(\frac{Z_\alpha^{-1} \e^{-\alpha\abs{x}}}{p^\omega_t(x)} \Big) + \log(Z_\alpha) + \alpha E_0^\omega[\abs{Z_t}],
    \end{align} and choosing \(\alpha = (\sqrt{t} + E_0^\omega[\abs{Z_t}])^{-1}\) yields with Jensen's inequality
    \begin{align}
        H^\omega(t)
        \leq \log(Z_\alpha) + 1.
    \end{align}
    It remains to bound, \(0 < \alpha \leq 1\),
    \begin{align}
        Z_\alpha
        = \sum_{x \in \Z^d} \e^{-\alpha \abs{x}}
        \lesssim_d \alpha^{-d}
    \end{align} to obtain \eqref{equation:entropy_comparison}.
\end{proof}
Still left to show is the activity bound used in the proof of Lemma~\ref{lemma:auxiliary_RW_diffusive_range} that we claimed in~\eqref{equation:activity_bound}:
\begin{proof}[Proof of Lemma~\ref{lemma:almost_sure_activity_bound}]
Let $\theta_{Z_t}\omega$ be the process of the environment as seen from the particle, cf.~\cite[Sec. 2.1.]{Biskup2011Recent}, and denote by $T_t$ its Markov semigroup, i.e. $T_tf=E_0^\omega[f(\theta_{Z_t}\omega)]$. The probability measure $\mathbb{Q}$ given by $m$-tilting $\mathbb P$, i.e. $\mathbb{Q} (\d \omega)=\frac{m^\omega(0)}{\E[m^\omega(0)]} \mathbb{P}(\d \omega)$, is invariant with respect to $(T_t)_t$. Since the jump rate of $Z$ is given by $a(\theta_{Z_t}\omega)$ for $a(\omega):=\frac{\mu^\omega(0)}{m^\omega(0)}$, we have
\begin{align*}
    E^{\omega}_0[N_t] = \int_{0}^{t} T_s a(\omega) \d s
\end{align*}
and therefore, simply by monotonicity in time and by the semigroup property,
\begin{align*}
    E^{\omega}_0[N_t] \leq \sum^{\lceil t \rceil - 1}_{n=0} T_1^n\left(\int_{0}^{1}T_s a(\omega) \d s\right).
\end{align*}
Notice that $\int_{0}^{1}T_s a(\omega) \d s\in L^1(\mathbb{Q})$ and $T_1$ is a positive Dunford-Schwartz operator, i.e., a positive contraction on both $L^1(\mathbb{Q})$ and $L^\infty(\mathbb{Q})$. Thus, by the Dunford-Schwartz pointwise ergodic theorem~\cite[Theorem 11.4]{EFHN2015}, there exists $K=K^\omega<\infty$ such that $\mathbb{Q}$-a.s.
\begin{align*}
    \sup_{t>0}\frac{1}{\lceil t \rceil}\sum^{\lceil t \rceil - 1}_{n=0} T_1^n\left(\int_{0}^{1}T_s a(\omega) \d s\right)\leq K^\omega<\infty.
\end{align*}

The claim now follows since $\mathbb Q$ and $\mathbb P$ are equivalent.
\end{proof}

\subsection{Confined path currents}\label{section:confined_path_currents}

Let us finally provide a small
\begin{proof}[Proof of Lemma~\ref{lemma:auxiliary_RW_confined_currents}]
    Write \(Q = Q_{\lceil L \sqrt{s} \rceil}\). We have 
    \begin{align}
        P_0^\omega(Z_{s/2} = x, \tau_Q > s/2)
        \lesssim_d s^{-d/2} m^\omega(x)
    \end{align} by Lemma~\ref{lemma:uniform_smoothing_heat_kernel_density_bound}.
    Applying the Markov property at \(s/2\) and then Lemma~\ref{lemma:fluctuations_energy_for_reflected_walk_in_equilibrium}, we get
    \begin{align}
        &E_0^\omega \Big[\mathbf{1}_{\tau_{Q} > s}  \sup_{s/2 \leq v \leq s} \abs{J^\omega_{v}(r) - J^\omega_{s/2}(r)}^2 \Big]
        \lesssim_d s^{-d/2} \sum_{x \in Q} m^\omega(x) E^\omega_x[\mathbf{1}_{\tau_Q > s/2} \sup_{v \leq s/2} \abs{J^\omega_{v}(r)}^2 ] \nonumber\\
        &\leq s^{-d/2} \sum_{x \in Q} m^\omega(x) E_x^{\omega, \mathrm{reflected}}[\sup_{v \leq s/2} \abs{J^\omega_{v}(r)}^2 ]  
        \lesssim s^{1-d/2} \sum_{x, y \in Q} \omega(x, y) \abs{r(\omega, x, y)}^2 \nonumber \\
        &\lesssim_d s L^d \mathcal{M}_0(e^\omega_r). \nonumber
    \end{align} The \(L^1\)-bound in the statement of the lemma follows with Cauchy-Schwarz.
\end{proof}

\subsection*{Disclosure of AI use}
The concrete proof strategy is a product of discussion with and mostly the accomplishment of \texttt{ChatGPT 5.6 Sol} and \texttt{ChatGPT 6 Astra}. It was only later clarified and written down more clearly by us, in particular all of the writing is our own. We take full responsibility for the correctness of the arguments presented in this note.

\subsection*{Acknowledgments} 
We are grateful for interesting conversations with Sebastian Andres regarding this area of mathematics.
Moreover, we want to thank every \texttt{AI} that we know for always being so friendly and helpful to us.

\bibliographystyle{alpha}
\bibliography{references}

\end{document}